\documentclass[letterpaper, 10 pt, conference]{ieeeconf}

\IEEEoverridecommandlockouts

\usepackage[utf8]{inputenc}
\usepackage{textcomp}

\usepackage{graphicx}
\usepackage[table]{xcolor}
\usepackage{hhline}
 
\usepackage{amsmath,amsthm,amssymb}
\usepackage{mathtools}
\usepackage[version=4]{mhchem}
\usepackage{siunitx}
\usepackage{longtable,tabularx}
\usepackage{booktabs}
\usepackage{multirow}
\usepackage{dsfont}
\usepackage[linesnumbered,ruled,vlined]{algorithm2e}

\newtheorem{define}{Definition}
\newtheorem{theorem}{Theorem}
\newtheorem{lemma}{Lemma}
\newtheorem{corr}{Corollary}
\newtheorem{remark}{Remark}
\newtheorem{assume}{Assumption}

\newcommand{\R}{\mathbb R}

\newcommand{\Nat}{\mathbb{N}}

\newcommand{\bmx}[1]{\begin{bmatrix}#1\end{bmatrix}} 
\newcommand{\bkt}[1]{\left[#1\right]} 
\newcommand{\pth}[1]{\left(#1\right)} 
\newcommand{\brc}[1]{\left \{#1\right \}} 
\newcommand{\nrm}[1]{\left \lVert#1\right \rVert} 

\newcommand{\ggeq}{\succeq} 

\DeclarePairedDelimiter{\ceil}{\lceil}{\rceil}
\DeclarePairedDelimiter{\floor}{\lfloor}{\rfloor}
\DeclarePairedDelimiter{\abs}{\lvert}{\rvert}

\makeatletter
\let\oldceil\ceil
\def\ceil{\@ifstar{\oldceil}{\oldceil*}}
\makeatother
\makeatletter
\let\oldfloor\floor
\def\floor{\@ifstar{\oldfloor}{\oldfloor*}}
\makeatother
\makeatletter
\let\oldnorm\norm
\def\norm{\@ifstar{\oldnorm}{\oldnorm*}}
\makeatother
\makeatletter
\let\oldabs\abs
\def\abs{\@ifstar{\oldabs}{\oldabs*}}
\makeatother

\newtheorem{problem}{Problem}
\newtheorem{prop}{Proposition}

\title{\LARGE \bf
Adversarial Robustness for Matrix Control Barrier Functions in Sampled-Data Systems
}

\author{James Usevitch%
\thanks{James Usevitch is with the Department of Electrical and Computer Engineering, Brigham Young University, Provo, UT, USA
        {\tt\small james\_usevitch@byu.edu}}%
}

\begin{document}

\maketitle
\thispagestyle{empty}
\pagestyle{empty}

\begin{abstract}

This paper presents novel theoretical results to guarantee multi-agent set invariance using Matrix Control Barrier Functions in sampled-data systems. More specifically, the paper presents conditions under which heterogeneous control-affine agents applying zero-order-hold control inputs can compute control inputs to render safe sets defined by matrix inequalities forward invariant. It then introduces methods to guarantee set invariance while accounting for the presence of adversarial agents seeking to drive the system state to unsafe sets. Finally, the paper presents theoretical extensions of these set invariance results to systems having high relative degree with respect to the matrix-valued safe set function.

\end{abstract}

\section{Introduction}

Set invariance and safety guarantees are critically important to the development and deployment of modern autonomous systems and robotics technologies.
Safety is commonly defined in the literature as forward invariance of a subset of the system's state space, and a fundamental problem in robotics and control is to derive feedback control laws that guarantee this property.
Control Barrier Functions (CBFs) \cite{ames2019control, garg2024advances} are a well-established class of techniques to derive both theoretical conditions and algorithmic control laws to ensure forward invariance of safe subsets of state space. CBFs have been successfully applied to a broad variety of domains including legged robots \cite{agrawal2017discrete}, aircraft \cite{so2024train}, optimization algorithms \cite{molnar2025collision}, satellites \cite{breeden2023robust}, self-driving vehicles \cite{chen2017obstacle}, and more.

Prior work on CBFs have typically defined safe sets in terms of superlevel sets of scalar-valued functions. Recent work has generalized CBF theory to matrix-valued functions through the introduction of Matrix CBFs (MCBFs) for both continuous-time \cite{ong2025matrix} and discrete-time \cite{usevitch2025computing} systems. These MCBFs define safe sets in terms of subsets of state space in which symmetric matrix-valued functions yield positive definite outputs. This generalization has significantly broadened the types of safety constraints that can be expressed in CBF form including direct constraints on matrix eigenvalues, encoding spectrahedra-based safe sets, and enforcing combinatorial \(p\)-choose-\(r\) sets of constraints \cite{ong2025combinatorial}.

Despite these advancements, there remain several limitations to prior work on MCBFs. First, no prior literature exists that considers MCBF methods for sampled-data systems. Such systems involve continuous-dynamics and a piecewise-constant control law, and are frequently more realistic models when describing physical plants driven by computer-based control algorithms.
The application of CBF and set invariance methods to nonlinear sampled-data systems is nontrivial since safety guarantees must hold for the entire sampling interval on which the control input is constant and not just for the time instants at which control is computed. Prior literature has presented methods to apply scalar-valued CBFs to sampled-data systems \cite{breeden2021control, usevitch2022adversarial}, but none have yet considered the more general class of matrix-valued CBFs.
Second, the limited prior work that exists on multi-agent MCBFs assumes that all agents are behaving cooperatively; i.e., all agents are collaboratively applying nominal control laws to ensure satisfaction of sufficient conditions for set invariance.
Modern multi-agent systems are increasingly exposed to faulty and adversarial misbehavior, either originating from their environment or from other agents. Prior literature has analyzed robustness to adversarial misbehavior for scalar-valued CBFs \cite{usevitch2022adversarial, usevitch2021adversarial}, but no work has generalized this to matrix-valued CBFs.
Finally, prior work on MCBFs has only considered systems with relative degree 1 with respect to the safe set function. Higher relative degrees have been considered for scalar-valued CBFs \cite{tan2021high, xiao2021high}, but this direction has not yet been generalized to the matrix-valued case.

To address these gaps, this paper presents novel methods to provide set invariance for heterogeneous, multi-agent, sampled-data dynamical systems in the presence of adversarial misbehavior using Matrix Control Barrier Function methods. To the best of our knowledge, this paper presents the first set invariance results using MCBFs for sampled-data systems, the first results guaranteeing adversarial robustness for MCBFs, and the first set invariance analysis for systems having high relative degree with respect to the matrix-valued safe set function.

This paper's specific contributions are as follows:
\begin{itemize}
    \item We present a novel set invariance result for matrix control barrier functions under sampled-data dynamics in cooperative multi-agent settings.
    \item We present a novel adversarially robust set invariance result for matrix control barrier functions under sampled-data dynamics.
    \item We present a novel set invariance result for systems with high relative degree using matrix control barrier functions. This result is valid for either cooperative or adversarial settings.
\end{itemize}

\section{Notation and Problem Formulation}
\label{sec:notation}

The notation \(H \in \mathcal{C}_{loc}^{1,1}\) denotes a continuously differentiable function \(H\) whose Jacobian \(\frac{\partial H}{\partial x}\) is locally Lipschitz continuous on its domain.
The set of symmetric, real-valued matrices of size \(p \times p\) is denoted \(\mathbb{S}^p\). The cone of positive definite symmetric \(p \times p\) matrices is denoted \(\mathbb{S}^p_++\) and the cone of positive definite symmetric matrices is denoted \(\mathbb{S}^p_+\).
Given matrices \(A,B \in \R^{n\times n}\), the notation \(A \ggeq 0\) and \(B \succ 0\) indicate that \(A\) is positive semidefinite and \(B\) is positive definite as per the Loewner order.
The notation \(\nrm{\cdot}\) indicates the Euclidean norm when applied to vectors in \(\R^n\) and the spectral norm when applied to matrices in \(\R^{m \times n}\). Given a vector \(v \in \mathcal{V}_1\) and a linear operator \(\mathcal{A} : \mathcal{V}_1 \to \mathcal{V}_2\), where \(\mathcal{V}_1\) and \(\mathcal{V}_2\) are normed vector spaces with respective norms \(\nrm{\cdot}_\alpha\) and \(\nrm{\cdot}_\beta\), the induced norm \(\nrm{\cdot}_{\alpha,\beta}\) is defined as (\cite{moon2000mathematical})
\begin{align*}
    \nrm{\mathcal{A}}_{\alpha,\beta} \triangleq \sup_{v \neq 0} \frac{\nrm{\mathcal{A} v}_\beta}{\nrm{v}_\alpha}.
\end{align*}
It follows that \(\nrm{\mathcal{A}v}_\beta \leq \nrm{\mathcal{A}}_{\alpha,\beta} \nrm{v}_\alpha\). When the specific values are omitted, it is assumed that \(\alpha = \beta = 2\).

\subsection{Problem Formulation}

Consider a system of \(N\) agents with indices \(\{1,\ldots,N\}\). The \(i\)th agent has state \(x^i \in \R^{n^i}\) and control input \(u^i \in \R^{m^i}\) with control affine dynamics
\begin{align}
    \dot{x}^i &= f^i(x^i) + g^i(x^i) u^i.
\end{align}
Each \(f^i : \R^{n^i} \to \R^{n^i}\) and \(g^i : \R^{n^i} \to \R^{n^i \times m^i}\) may differ among agents, but all are assumed to be locally Lipschitz in their arguments on their respective domains.
Agents' control inputs \(u^i\) are updated in a sampled-data manner using a zero-order hold (ZOH) method. More specifically, each agent applies a constant control input \(u^i(t) = u^i_{k}\) \(\forall t \in [t_k, t_{k+1})\) where the sampling times are defined as \(t_k \triangleq t_0 + k \Delta_t\), \(k \in \mathbb{N}\), \(\Delta_t > 0\). The set of sampling times is denoted
\(\mathcal{T} = \{t_k\}_{k \in \mathbb{N}}\).
It is assumed that agents receive knowledge of the system state \(\vec{x}\) at sampling times \(t_k\) and may apply feedback control laws \(u^i_{k} = \kappa^i(\vec{x}(t_k))\).

For brevity, we define \(n \triangleq \sum_{i=1}^N n^i\) and \(m = \sum_{i=1}^N m^i\).
The vector of combined states for all agents is denoted \(\vec{x} \triangleq \bmx{(x^1)^\intercal & \cdots & (x^N)^\intercal}^\intercal \in \R^{n}\), the vector of combined control inputs is denoted \(\vec{u} \triangleq \bmx{(u^1)^\intercal & \cdots & (u^N)^\intercal}^\intercal \in \R^{m}\).
For brevity, we will denote \(\vec{x}_k \triangleq \vec{x}(t_k)\) and \(\vec{u}_k \triangleq \vec{u}(t_k)\).
The dynamics of the combined system is denoted
\begin{align}
    \dot{\vec{x}} = F(\vec{x}) + G(\vec{x})\vec{u}_k = \bmx{f^1(x^1) + g^1(x^1)u^1_k \\ \vdots \\ f^N(x^N) + g^N(x^N)u^N_k}  \label{eq:multi_agent_dynamics}
\end{align}
where \(F(\vec{x}) \triangleq \bmx{(f^1(x^1))^\intercal & \cdots (f^N(x^N))^\intercal}^\intercal\) and \(G(\vec{x}) \triangleq \text{blkdiag}\pth{g^1(x^1), \cdots, g^N(x^N)}\).
Since each control input is a piecewise constant function of time, Carath\'eodory's theorem guarantees existence and uniqueness of solutions to the combined system dynamics \eqref{eq:multi_agent_dynamics} \cite[Sec. 2.2]{grune2016nonlinear}, \cite[Thm. 54]{sontag2013mathematical}.

Each agent's control input \(u^i\) is constrained to lie within a nonempty convex, compact subset \(\mathcal{U}^i \subset \R^{m^i}\).
The Cartesian product of these subsets is denoted \(\mathcal{U} = \bigtimes_{i=1}^N \mathcal{U}^i\), with \(\vec{u} \in \mathcal{U}\).
Observe that \(\mathcal{U}\) is compact since each \(\mathcal{U}^i\) is compact, which implies that the norm \(\nrm{\vec{u}}\) has the upper bound
\begin{align}
    b_u \triangleq \sup_{\vec{u} \in \mathcal{U}} \nrm{\vec{u}}. \label{eq:b_u}
\end{align}
For purposes of computational feasibility, each subset \(\mathcal{U}^i\) is assumed to be the intersection of a finite number of convex cones. This will ensure that the constraint \(u^i \in \mathcal{U}^i\) can be handled by state-of-the-art convex optimization solvers \cite{Clarabel_2024, mosek}.

The state space for the combined system is partitioned into a safe set \(\mathcal{S} \subset \R^{n}\) and an unsafe set \(\overline{\mathcal{S}} = \R^{n} \backslash \mathcal{S}\). The safe set \(\mathcal{S}\) is defined in terms of the superlevel sets of a matrix-valued function \(H : \R^{n} \to \mathbb{S}^p\) as follows:
\begin{align}
    S &= \brc{\vec{x} \in \R^{n} : H(\vec{x}) \ggeq 0}. \label{eq:safe_set}
\end{align}
We assume \(H \in \mathcal{C}_{loc}^{1,1}\); i.e., \(H\) is continuously differentiable and its (tensor-valued) Jacobian \(\frac{\partial H}{\partial x}\) is locally Lipschitz continuous on its domain.

We will be mainly interested in ensuring that the set \(\mathcal{S}\) is rendered forward invariant, which is defined as follows.
\begin{define}
    A set \(\mathcal{C} \subset \R^{n}\) is forward invariant with respect to the dynamics \eqref{eq:multi_agent_dynamics} if \(\vec{x}(t) \in \mathcal{C}\) for all \(t \geq t_0\).
    The set \(\mathcal{C}\) is called \emph{safe} if it is both forward invariant and \(\mathcal{C} \subseteq \mathcal{S}\).
\end{define}

The central problem of this paper is as follows:
\begin{problem}
    Given the multi-agent sampled-data dynamics \eqref{eq:multi_agent_dynamics}, derive a method to compute control inputs \(u^i_k\) rendering sets of the form \eqref{eq:safe_set} forward invariant.
\end{problem}
This problem will be solved for three settings: 1) under the assumption of cooperative behavior among all agents, 2) in the presence of adversarial misbehavior from a subset of agents, and 3) for systems with high relative degree. The precise definition of adversarial behavior will be given in Section \ref{sec:adversarial}.

\subsection{Preliminaries and Overview of Matrix Control Barrier Functions}

In preparation for the main results we briefly review several relevant preliminary concepts.
Similar to \cite{ong2025matrix}, the entry-wise Lie derivative matrix \(\mathcal{L}_F H : \R^n \to \mathbb{S}^p\) along a vector field \(F\) is defined as
\begin{align}
    [\mathcal{L}_F H]_{ij}(\vec{x}) = \mathcal{L}_F (H_{ij}) (\vec{x}),\ \forall i, j \in \{1,\ldots,p\}.
\end{align}
Next, let \(G_j\) be the \(j\)th column of \(G\) and let \([\vec{u}]_j\) be the \(j\)th entry of \(\vec{u}\). The entrywise Lie derivative \(\mathcal{L}_{G_j} H(\vec{x}) : \R^m \to \mathbb{S}^p\)
operates on the control input \(\vec{u}\) as follows:
\begin{align*}
    (\mathcal{L}_G H(\vec{x}))(\vec{u}) \triangleq \sum_{j=1}^m \mathcal{L}_{G_j} H(x) [\vec{u}]_j \in \mathbb{S}^p.
\end{align*}

It follows that \(\dot{H}(\vec{x}) = \mathcal{L}_{\widehat{F}} H(\vec{x})\) for a (possibly closed-loop) autonomous system \(\dot{\vec{x}} = \widehat{F}(\vec{x})\).
Likewise, \(\dot{H}(\vec{x}, \vec{u}) = \mathcal{L}_F H(\vec{x}) + (\mathcal{L}_G H(\vec{x})) (\vec{u})\) for the control affine dynamics in \eqref{eq:multi_agent_dynamics}.

The following result from \cite{ong2025matrix} establishes a sufficient condition for forward invariance of the set \(S\) for continuous-time autonomous systems.
\begin{prop}[{Adapted from \cite{ong2025matrix}}]
    \label{prop:mcbf}
    Consider the autonomous system with control affine dynamics
    \(\dot{\vec{x}} = \widehat{F}(\vec{x})\)
    with a continuous vector field \(\widehat{F}\)
    Let \(S\) be defined as in \eqref{eq:safe_set}. If the following barrier condition holds with a positive constant \(c_\alpha > 0\),
    \begin{align}
        \dot{H}(\vec{x}) \ggeq -c_\alpha H(\vec{x})\label{eq:orig_mcbf_safety_cond}
    \end{align}
    for all \(\vec{x}\) in an open neighborhood \(\mathcal{E} \supset S\), then set \(S\) is forward invariant for the system.
\end{prop}
This result motivated the definition of \emph{Exponential MCBFs}, which can be found in \cite{ong2025matrix}. However, Exponential MCBFs assume that the system applies a continuous (i.e., non-ZOH) control input \(\vec{u}(\vec{x}(t))\).
When applying a closed-loop, ZOH control input \(\vec{u}_k = \vec{\kappa}(\vec{x}(t_k))\) over the time interval \(t \in [t_k, t_{k+1})\), care must be taken so that the inequality in \eqref{eq:orig_mcbf_safety_cond} is not violated as the state \(\vec{x}(t)\) evolves but the control input \(\vec{u}_k\) remains piecewise constant.

Finally, the following theorem known as \emph{Weyl's Inequality} will be used to obtain lower bounds on the eigenvalues of matrices.
\begin{prop}[{Weyl's Inequality \cite[Thm 4.3.1]{horn2012matrix}}]
    \label{thm:weyl}
    Let \(A, B \in \mathbb{S}^p\). Let the respective eigenvalues of \(A,B\) and \(A+B\) be \(\{\lambda_i(A)\}_{i=1}^p\), \(\{\lambda_i(B)\}_{i=1}^p\), and \(\{\lambda_i(A+B)\}_{i=1}^p\). Let the eigenvalues be ordered such that \(\lambda_1 \leq \lambda_2 \leq \cdots \leq \lambda_p\). Then the following holds for all $j \in \{1, \ldots, i\}$:
    \begin{align}
        \lambda_{i-j+1}(A) + \lambda_{j}(B) \leq \lambda_i(A+B).
    \end{align}
\end{prop}
Given two symmetric matrices \(A, B \in \mathbb{S}^p\), the following corollary establishes lower bounds on their eigenvalues in terms of the spectral norm of their difference.
\begin{corr}
    \label{eq:Weyl_corollary}
    Let \(A, B \in \mathbb{S}^p\). Let the eigenvalues of each matrix be ordered as in Proposition \ref{thm:weyl}. Then the following holds:
    \begin{align}
        \lambda_1(A) \geq \lambda_1(B) - \nrm{A - B}_2.
    \end{align}
\end{corr}
\begin{proof}
    Recall that for symmetric matrices, \(\nrm{A}_2 = \sqrt{\lambda_p(A^2)} = |\lambda_p(A)|\). By the ordering of eigenvalues, \(|\lambda_1(A)| \leq |\lambda_p(A)|\), which implies \(-|\lambda_p(A)| \leq \lambda_1(A) \leq |\lambda_p(A)|\).
    Let \(E = A - B\). By Weyl's inequality with \(i = j = 1\), we have
    \begin{align*}
        \lambda_1(B + E) \geq \lambda_1(B) + \lambda_1(E) \geq \lambda_1(B) - \nrm{E}_2.
    \end{align*}
    The result follows by substituting in the definition of \(E\).
\end{proof}

\section{Main Results}
\label{sec:main}

We now present our main results. Subsection \ref{sec:cooperative} gives a set invariance result for sampled-data systems in cooperative multi-agent settings. Next, subsection \ref{sec:adversarial} set invariance conditions in the presence of adversaries are derived. Finally, subsection \ref{sec:high_order} presents set invariance conditions for systems with high relative degree with respect to the safe set function.

\subsection{Zero-Order Hold Exponential Matrix Control Barrier Functions}
\label{sec:cooperative}

Let \(D \subseteq \R^n\) be a domain such that \(S \subset D\).
The following standing assumptions are made:

\begin{assume}
    \label{assume:Lipschitz_D}
    The function \(H\) and Lie derivatives \(\mathcal{L}_F H(x)\) and \(\mathcal{L}_G H(x)\) are Lipschitz on \(D\) with respective constants \(L_H, L_{FH}, L_{GH} \geq 0\).\footnote{As a special case, this holds when \(H\), \(\mathcal{L}_F H(x)\) and \(\mathcal{L}_G H(x)\) are locally Lipschitz and \(D\) is compact.}
\end{assume}

\begin{assume}
    \label{assume:norm_bounds}
    There exist finite upper bounds \(b_F, b_G > 0\) such that \(\sup_{\vec{x}\in D} \nrm{F(\vec{x})}_2 \leq b_F\) and \(\sup_{\vec{x} \in D} \nrm{G(\vec{x})}_2 \leq b_G\).\footnote{This likewise holds as a special case when \(D\) is compact.}
\end{assume}

\subsection{Multi-Agent Exponential Matrix Control Barrier Functions}

To guarantee forward invariance of \(S\) under the dynamics \eqref{eq:multi_agent_dynamics}, we must ensure that the piecewise constant control input \(\vec{u}_k = \vec{\kappa}(\vec{x}(t_k))\) is chosen such that \(H(x(t)) \ggeq 0\) for all times \(t \in [t_k, t_{k+1})\) on each sampling interval.
Towards this end, we define the function
\begin{align}
    \Psi(\vec{x}, \vec{u}) &\triangleq \dot{H}(\vec{x}, \vec{u}) + c_\alpha H(\vec{x}), \nonumber \\
    &= \mathcal{L}_F H(\vec{x}) + \mathcal{L}_G H(\vec{x}) \vec{u} + c_\alpha H(\vec{x}), \label{eq:sub_dynamics}
\end{align}
where \eqref{eq:sub_dynamics} follows from evaluating \(\dot{H}\) under the system dynamics \eqref{eq:multi_agent_dynamics}.
Observe by Proposition \ref{prop:mcbf} that the existence of a constant \(\vec{u}_k\) such that \(\Psi(\vec{x}(t), \vec{u}_k) \ggeq 0\) \(\forall t \in [t_k, t_{k+1})\) is a sufficient condition for ensuring \(H(x(t)) \ggeq 0\) and therefore \(\vec{x}(t) \in S\) \(\forall t \in [t_k, t_{k+1})\). We accordingly seek to analyze and bound the worst-case normed error between the value of \(\Psi(\vec{x}_k, \vec{u}_k)\) at the sampled time when the control input is computed, and the value of \(\Psi(\vec{x}(t), \vec{u}_k)\) as the state evolves over the sampling interval:
\begin{align*}
    E(t_k, \vec{x}_k, \vec{u}_k) &\triangleq  \sup_{t \in [t_k, t_{k+1})} \nrm{\Psi(\vec{x}(t), \vec{u}_k) - \Psi(\vec{x}_k, \vec{u}_k)}_2.
\end{align*}
Computing \(E\) directly is intractable in general, but its value can be bounded as follows.
Let \(D \subset \R^n\) be the domain from Assumptions \ref{assume:Lipschitz_D} and \ref{assume:norm_bounds}.
From \eqref{eq:multi_agent_dynamics} we can derive the following bound on the norm of the state evolution over any interval \(t \in [t_k, t_{k+1})\):
\begin{align}
    &\nrm{\vec{x}(t) - \vec{x}_k}_2 = \nrm{\int_{t_k}^{t} \pth{F(\vec{x}(\tau)) + G(\vec{x}(\tau)) \vec{u}_k} d\tau}_2, \nonumber \\
    & \leq \int_{t_k}^{t_{k+1}}
    \sup_{\substack{\vec{x} \in D \\ \vec{u} \in \mathcal{U}}}
    \nrm{F(\vec{x}) + G(\vec{x}) \vec{u}_k}_2 d\tau, \nonumber\\
    &\leq \Delta_t \bkt{\pth{\sup_{\vec{x} \in D} \nrm{F(\vec{x})}_2} + \pth{\sup_{\vec{x} \in D} \nrm{G(\vec{x})}_2} \pth{\sup_{\vec{u} \in \mathcal{U}} \nrm{\vec{u}}_2}},
    \label{eq:state_norm_bound}
\end{align}
where we have used the fact that \(t_{k+1} - t_k = \Delta_t\).
Observe that this upper bound is independent of \(t\) and therefore  holds for all \(t \in [t_k, t_{k+1})\).
Using \eqref{eq:b_u} and Assumption \ref{assume:norm_bounds}, equation \eqref{eq:state_norm_bound} can therefore be rewritten as
\begin{align}
    \nrm{\vec{x}(t) - \vec{x}_k}_2 \leq \Delta_t(b_F + b_G b_u)\ \forall t \in [t_k, t_{k+1}), \label{eq:bound_on_x}
\end{align}
where \(\Delta_t\) is the sampling interval.

\begin{remark}
    It may be possible to tighten the bound in \eqref{eq:bound_on_x} by considering the quantity \(\sup_{\substack{\vec{x} \in D \\ \vec{u} \in \mathcal{U}}}
    \nrm{F(\vec{x}) + G(\vec{x}) \vec{u}_k}_2\) instead of the sum of the norms of its individual components. However, in practice it may be more straightforward to obtain bounds on the norms of the components rather than a bound on the norm of the combined expression.
\end{remark}

Using this result, the following Lemma demonstrates an upper bound on the value of \(E(t_k, \vec{x}_k, \vec{u}_k)\).

\begin{lemma}
    \label{lemma:E_bound}
    The following holds for all \(\vec{x}_k \in D\) and \(\vec{u}_k \in \mathcal{U}\):
    \begin{align*}
        E(t_k, \vec{x}_k, \vec{u}_k) \leq \Delta_t (L_{FH} + L_{GH}b_u + c_\alpha L_H) (b_F + b_G b_u),
    \end{align*}
    where \(b_u\) is defined in \eqref{eq:b_u}, \(L_{FH}, L_{GH}, L_{H}\) are defined in Assumption \ref{assume:Lipschitz_D}, and \(b_F, b_G\) are defined in Assumption \ref{assume:norm_bounds}.
\end{lemma}

\begin{proof}
    Observe that
    \begin{align*}
        &\nrm{\Psi(\vec{x}(t), \vec{u}_k) - \Psi(\vec{x}_k, \vec{u}_k)}_2 \leq  \nrm{\mathcal{L}_F H(\vec{x}(t)) - \mathcal{L}_F H(\vec{x}_k)}_2 \\
        & \hspace{1em} +
        \nrm{\mathcal{L}_G H(\vec{x}(t)) - \mathcal{L}_G H(\vec{x}_k)}_2 \nrm{\vec{u}_k}_2 \\
        &\hspace{1em} +
        c_\alpha \nrm{H(\vec{x}(t)) - H(\vec{x}_k)}_2.
    \end{align*}
    By Assumption \ref{assume:Lipschitz_D}, we have
\begin{align}
    \begin{aligned}
    &\nrm{\Psi(\vec{x}(t), \vec{u}_k) - \Psi(\vec{x}_k, \vec{u}_k)}_2 \leq \\
    &\hspace{2em}(L_{FH} + L_{GH} b_u + c_\alpha L_{H})\nrm{\vec{x}(t) - \vec{x}_k}_2.
    \end{aligned}
\end{align}
Substituting in \eqref{eq:bound_on_x} yields
\begin{align}
    \begin{aligned}
    &\nrm{\Psi(\vec{x}(t), \vec{u}_k) - \Psi(\vec{x}_k, \vec{u}_k)}_2 \leq \\
    &\hspace{2em}\Delta_t (L_{FH} + L_{GH} b_u + c_\alpha L_{H}) (b_F + b_G b_u),
    \end{aligned}
\end{align}
which concludes the proof.
\end{proof}
For brevity, we define
\begin{align*}
    \delta(\Delta_t, c_\alpha) \triangleq \Delta_t (L_{FH} + L_{GH} b_u + c_\alpha L_{H}) (b_F + b_G b_u),
\end{align*}
and write \(E(t_k, \vec{x}_k, \vec{u}_k) \leq \delta(\Delta_t, c_\alpha)\).

Given this bound, we are now prepared to present a novel extension of Exponential MCBFs to multi-agent, zero-order hold settings.

\begin{define}
    \label{def:ZE_MCBF}
     Let \(H : \R^n \to \mathbb{S}^p\), \(H \in \mathcal{C}_{loc}^{1,1}\) be a matrix valued function and let \(S\) be defined as in \eqref{eq:safe_set}.
     Let \(D \subseteq \R^n\) be a domain such that \(S \subset D\).
     The function \(H\) is called a Zero-order-hold Exponential Matrix Control Barrier Function (ZE-MCBF) for the dynamics \eqref{eq:multi_agent_dynamics} if there exists a constant \(c_\alpha > 0\) such that for all \(\vec{x} \in D\) there exists a \(\vec{u} \in \mathcal{U}\) satisfying
    \begin{align}
        \mathcal{L}_F H(\vec{x}) + (\mathcal{L}_G H(\vec{x}))(\vec{u}) \ggeq -c_\alpha H(\vec{x}) + \delta(\Delta_t, c_\alpha)I. \label{eq:ZEMCBF_cond}
    \end{align}
\end{define}

Define the set of feasible control inputs satisfying \eqref{eq:ZEMCBF_cond} as
\begin{align}
    &K(\vec{x}) \triangleq \brc{\vec{u} \in \mathcal{U} : \Psi(\vec{x}, \vec{u}) \ggeq \delta(\Delta_t, c_\alpha) I}, \label{eq:Kset} \\
    &= \big\{ \vec{u} \in \mathcal{U} : \mathcal{L}_F H(\vec{x}) + (\mathcal{L}_G H(\vec{x}))(\vec{u}) \nonumber \\
    &\hspace{6em} + c_\alpha H(\vec{x}) \ggeq \delta(\Delta_t, c_\alpha)I \big\}. \nonumber
\end{align}
Observe that this is equivalent to requiring the smallest eigenvalue to satisfy \(\lambda_1\pth{\dot{H}(\vec{x}, \vec{u}) + c_\alpha H(\vec{x})} \geq \delta(\Delta_t, c_\alpha)\).

\begin{theorem}
    \label{thm:ZE-MCBF}
    Consider the dynamics \eqref{eq:multi_agent_dynamics} and the set \(S\) in \eqref{eq:safe_set}. If \(H\) is a ZE-MCBF, then any feedback control law \(\vec{\kappa}(\cdot)\) satisfying \(\vec{\kappa}(\vec{x}_k) \in K(\vec{x}_k)\) for all \(t_k\), \(k \in \Nat\) renders \(S\) forward invariant for the closed-loop system.
\end{theorem}

\begin{proof}
    Let \(\vec{x}(t_0) \in S\).
    Choose any feedback control law such that \(\vec{\kappa}(\vec{x}_k) \in K(\vec{x}_k)\) for all \(k \in \Nat\).
    From Lemma \ref{lemma:E_bound} we have \(E(t_k, \vec{x}_k, \vec{u}_k) \leq \delta(\Delta_t, c_\alpha)\) for all \(t \in [t_k, t_{k+1})\).
    By Corollary \ref{eq:Weyl_corollary} and the definition of \(K(\vec{x})\) in \eqref{eq:Kset}, the following holds for all \(t \in [t_k, t_{k+1})\):
    \begin{align*}
        &\lambda_1\pth{\Psi(\vec{x}(t), \vec{\kappa}(\vec{x}))} \geq  \\
        &\hspace{2em} \lambda_1\pth{\Psi(\vec{x}_k, \vec{\kappa}(\vec{x}_k))} - \nrm{\Psi(\vec{x}(t), \vec{\kappa}(\vec{x})) - \Psi(\vec{x}_k, \vec{\kappa}(\vec{x}_k))}_2, \\
        &\hspace{1em} \geq \lambda_1\pth{\Psi(\vec{x}_k, \vec{\kappa}(\vec{x}_k))} - E(t_k, \vec{x}_k, \vec{u}_k), \\
        &\hspace{1em} \geq  \lambda_1\pth{\Psi(\vec{x}_k, \vec{\kappa}(\vec{x}_k))} - \delta(\Delta_t, c_\alpha), \\
        &\hspace{1em} \geq \lambda_1\pth{\delta(\Delta_t, c_\alpha)I} - \delta(\Delta_t, c_\alpha), \\
        &\hspace{1em} \geq 0.
    \end{align*}
    It follows that \(\Psi(\vec{x}(t), \vec{\kappa}(\vec{x}_k)) = \dot{H}(\vec{x}(t), \kappa(\vec{x}_k)) + c_\alpha H(\vec{x}(t)) \ggeq 0\) for all \(t \in [t_k, t_{k+1})\).
    By Proposition \ref{prop:mcbf}, the set \(S\) is forward invariant on the interval \([t_k, t_{k+1})\).
    Since this holds for all \(k \in \Nat\), the set \(S\) is therefore forward invariant for all \(t \geq t_0\).
\end{proof}
We point out that the ZOH feedback law \(\vec{\kappa}(\cdot)\) is not required to be continuous, since existence and uniqueness of solutions for the sampled-data system are guaranteed by Carath\'eodory's theorem as noted previously. An example of a feedback law satisfying \(\vec{\kappa}(\vec{x}_k) \in K(\vec{x}_k)\) is given by the following parametric convex conic program, which minimally modifies a nominal control input \(\vec{u}_\text{nom}\):
\begin{alignat}{2}
        &\vec{\kappa}\pth{\vec{x}} \triangleq \min_{\vec{u} \in \R^m} \hspace{1em} \nrm{\vec{u} - \vec{u}_\text{nom}}_2^2 \label{eq:kappa_convex_program}\\
        & \text{s.t.} \hspace{1em} \mathcal{L}_F H(\vec{x}) + (\mathcal{L}_G H(\vec{x}))(\vec{u}) \ggeq -c_\alpha H(\vec{x}) + \delta(\Delta_t, c_\alpha)I, \nonumber \\
        &\hspace{2em}\vec{u} \in \mathcal{U}. \nonumber
\end{alignat}
Observe that the objective is quadratic and convex in \(\vec{u}\), the first constraint is a convex semidefinite constraint in \(\vec{u}\), and the second constraint is a convex conic constraint.

\subsection{Adversarial Robustness for Exponential Matrix Control Barrier Functions}
\label{sec:adversarial}

We next consider the presence of adversarial agents. The definition of adversarial agents is as follows:
\begin{define}
    An adversarial agent with index \(j \in \{1,\ldots,N\}\) is an agent that applies the following control input for all sampling times \(t_k,\ k \in \Nat\):
    \begin{align}
        u^j_{\min}(\vec{x}) \triangleq \underset{u \in \mathcal{U}_j}{\arg\min} \bkt{\mathcal{L}_{g^j} H(\vec{x}) u} \label{eq:adversarial_u}
    \end{align}
\end{define}
In words, adversarial agents apply maximum possible control authority towards attempting to violate the set invariance condition in \eqref{eq:orig_mcbf_safety_cond}.
Observe that each \(u^j_{\min}(\vec{x})\) can be computed via convex optimization methods since the objective and constraints are both convex in the optimization variable.

The set of adversarial agent indices is denoted \(\mathcal{A} \subset \{1,\ldots,N\}\). Agents that are not adversarial are called \emph{normal}. The set of normal agent indices is denoted \(\mathcal{N} = \{1,\ldots,N\} \backslash \mathcal{A}\).
We denote
the vectors of normal and adversarial control inputs as \(\vec{u}^{\mathcal{N}} \triangleq \bmx{(u^{i_1})^\intercal, \cdots, (u^{i_{|\mathcal{N}|}})^\intercal}^\intercal,\ i_q \in \mathcal{N}\) and
\begin{align}
    \label{eq:adv_vec_u}
    \vec{u}^{\mathcal{A}}(\vec{x}) \triangleq \bmx{\pth{u^{j_1}_{\min}(\vec{x})}^\intercal, \cdots, \pth{u^{j_{|\mathcal{A}|}}_{\min}(\vec{x})}^\intercal}^\intercal,
\end{align}
with \(j_q \in \mathcal{A}\).
The Cartesian products of normal and adversarial feasible control input sets are denoted \(\mathcal{U}^{\mathcal{N}} \triangleq \bigtimes_{i \in \mathcal{N}} \mathcal{U}^{i}\) and \(\mathcal{U}^{\mathcal{A}} \triangleq \bigtimes_{j \in \mathcal{A}} \mathcal{U}^{j}\).

We are interested in finding a feedback control law for the normal agents \(\vec{\kappa}^\mathcal{N}(\vec{x})\) that renders the set \(S\) forward invariant despite the adversarial behavior described in \eqref{eq:adversarial_u}.
In the presence of these adversarial control inputs, the expression for \(\dot{H}\) becomes
\begin{align}
    &\dot{H}(\vec{x}, \vec{u}^\mathcal{N}) = \mathcal{L}_F H(\vec{x}) + \label{eq:Hdot_adversarial}\\
    &\hspace{1em} \pth{\sum_{i \in \mathcal{N}} (\mathcal{L}_{g^i} H(\vec{x}))(u^{i})} +
     \pth{\sum_{j \in \mathcal{A}}  (\mathcal{L}_{g^j} H(\vec{x}) )
     (u^j_{\min}(\vec{x}))} \nonumber
\end{align}
Similar to the previous subsection, define the function
\begin{align}
    \Psi^\mathcal{N}(\vec{x}, \vec{u}^{\mathcal{N}}) \triangleq \dot{H}(\vec{x}, \vec{u}^\mathcal{N}) + c_\alpha H(\vec{x}). \label{eq:PsiN}
\end{align}
As before, the existence of a constant \(\vec{u}^\mathcal{N}_k\) such that \(\Psi(\vec{x}, \vec{u}^\mathcal{N}) \ggeq 0\) \(\forall t \in [t_k, t_{k+1})\) is a sufficient condition for ensuring \(H(x(t)) \ggeq 0\) \(\forall t \in [t_k, t_{k+1})\). We define the following maximum normed error between the value of \(\Psi\) at the sampled time \(t_k\) and the current time \(t \in [t_k, t_{k+1})\):
\begin{align*}
    E^\mathcal{N}(t_k, \vec{x}_k, \vec{u}_k^\mathcal{N}) &\triangleq  \sup_{t \in [t_k, t_{k+1})} \nrm{\Psi(\vec{x}(t), \vec{u}_k^\mathcal{N}) - \Psi(\vec{x}_k, \vec{u}_k^\mathcal{N})}_2.
\end{align*}

\begin{lemma}
    \label{lemma:adversarial_E}
    The following holds for all \(\vec{x}_k \in D\) and \(\vec{u}_k \in \mathcal{U}\):
    \begin{align*}
        E^\mathcal{N}(t_k, \vec{x}_k, \vec{u}_k) \leq \Delta_t (L_{FH} + L_{GH}b_u + c_\alpha L_H) (b_F + b_G b_u),
    \end{align*}
\end{lemma}

\begin{proof}
    Follows from the proof of Lemma \ref{lemma:E_bound} and the fact that \(\nrm{u^j_{\min}(\vec{x})}_2 \leq \sup_{u^j \in \mathcal{U}^j} \nrm{u^j}_2\) \(\forall \vec{x} \in \R^n\), \(\forall j \in \mathcal{A}\).
\end{proof}

The following definition extends the notion of ZE-MCBFs to incorporate robustness against adversarial actions.

\begin{define}
    \label{def:ARZE-MCBF}
    Let \(H : \R^n \to \mathbb{S}^p\), \(H \in \mathcal{C}_{loc}^{1,1}\) be a matrix valued function and let \(S\) be defined as in \eqref{eq:safe_set}.
    Let \(D \subseteq \R^n\) be a domain such that \(S \subset D\).
    The function \(H\) is called an Adversarially Robust Zero-order-hold Exponential Matrix Control Barrier Function (ARZE-MCBF) for the dynamics \eqref{eq:multi_agent_dynamics} if there exists a constant \(c_\alpha > 0\) such that for all \(\vec{x} \in D\) there exists a \(\vec{u}^\mathcal{N} \in \mathcal{U}^\mathcal{N}\) satisfying
    \begin{align}
        &\mathcal{L}_F H(\vec{x}) + \pth{\sum_{i \in \mathcal{N}} (\mathcal{L}_{g^i} H(\vec{x}))(u^{i})} + \label{eq:ZEMCBF_cond} \\
     &\hspace{1em }\pth{\sum_{j \in \mathcal{A}}  (\mathcal{L}_{g^j} H(\vec{x}))
     (u^j_{\min}(\vec{x}))} \ggeq - c_\alpha H(\vec{x}) +   \delta(\Delta_t, c_\alpha) I. \nonumber
    \end{align}
\end{define}

We define the set of normal agent feasible control inputs satisfying \eqref{eq:ZEMCBF_cond} as follows:
\begin{align}
    &\mathcal{K}^\mathcal{N}(\vec{x}) \triangleq \brc{\vec{u}^\mathcal{N} \in \mathcal{U}^\mathcal{N} : \Psi^\mathcal{N}(\vec{x}, \vec{u}^\mathcal{N}) \ggeq \delta(\Delta_t, c_\alpha) I}, \label{eq:Knormalset} \\
    &\hspace{1em} = \Bigg\{ \vec{u}^\mathcal{N} \in \mathcal{U} : \mathcal{L}_F H(\vec{x}) + \pth{\sum_{i \in \mathcal{N}} (\mathcal{L}_{g^i} H(\vec{x}))(u^{i})} + \nonumber \\
     &\pth{\sum_{j \in \mathcal{A}}  (\mathcal{L}_{g^j} H(\vec{x}))
     (u^j_{\min}(\vec{x}))} \ggeq - c_\alpha H(\vec{x}) +   \delta(\Delta_t, c_\alpha) I. \nonumber \Bigg\}. \nonumber
\end{align}

\begin{theorem}
    \label{thm:ARZE-MCBF}
    Consider the dynamics \eqref{eq:multi_agent_dynamics} and the set \(S\) in \eqref{eq:safe_set}. If \(H\) is a ARZE-MCBF, then any feedback control law for the normal agents \(\vec{\kappa}^\mathcal{N}(\cdot)\) satisfying \(\vec{\kappa}^\mathcal{N}(\vec{x}_k) \in K^\mathcal{N}(\vec{x}_k)\) for all \(t_k\), \(k \in \Nat\) renders \(S\) forward invariant for the closed-loop system.
\end{theorem}

\begin{proof}
    Follows from similar arguments as Theorem \ref{thm:ZE-MCBF}.
\end{proof}

Feasible control inputs within \(K^{\mathcal{N}}(\vec{x})\) can be computed using convex optimization methods similar to \eqref{eq:kappa_convex_program}.

\subsection{Systems with High Relative Degree}
\label{sec:high_order}

We now consider scenarios where the system \eqref{eq:multi_agent_dynamics} has relative degree greater than one with respect to the function \(H\).
To analyze the evolution of \(H\) and its time derivatives on each of the intervals \(t \in [t_k, t_{k+1})\), \(k \in \Nat\),
we define the following set of \(\Psi_q\) functions:
\begin{align}
    \Psi_0(\vec{x}) &\triangleq H(\vec{x}), \label{eq:Psi_high_degree}\\
    \Psi_q(\vec{x}) &= \dot{\Psi}_{q-1}(\vec{x}) + c_{\alpha_q}(\Psi_{q-1}(\vec{x})),\ 1 \leq q \leq r-1, \nonumber \\
    \Psi_{r}(\vec{x}, \vec{u}^\mathcal{N}) &\triangleq \dot{\Psi}_{r-1}(\vec{x}, \vec{u}^\mathcal{A}_{\min}(\vec{x}), \vec{u}^\mathcal{N}) + c_{\alpha_{r}} \Psi_{r-1}(\vec{x}), \nonumber \\
    &= \mathcal{L}_F \Psi_{r-1}(\vec{x}) + \pth{\sum_{i \in \mathcal{N}} (\mathcal{L}_{g^i} \Psi_{r-1}(\vec{x}))(u^{i})} + \nonumber \\
     & \hspace{-2em}\pth{\sum_{j \in \mathcal{A}}  (\mathcal{L}_{g^j} \Psi_{r-1}(\vec{x})) +
     (u^j_{\min}(\vec{x}))} + c_{\alpha_{r}} \Psi_{r-1}(\vec{x}). \nonumber
\end{align}
Here, the constants \(c_{\alpha_q}\) satisfy \(c_{\alpha_q} \geq 0\), \( 1 \leq q \leq r\).
The integer \(r \in \mathbb{Z}_\geq 1\) is the smallest integer such that control inputs \(\vec{u}^\mathcal{N}\) or \(\vec{u}^\mathcal{A}(\vec{x})\) appear in \(\Psi_{r}\) and are absent in \(\Psi_q\), \(0 \leq q \leq r -1\).
We further assume that \(\vec{u}^\mathcal{A}(\vec{x}), \vec{u}^\mathcal{N}\) both appear in the expression for \(\Psi_r\). Due to space constraints, extensions to more general settings are left for future work.

Similar to previous subsections, let the error function \(E_r^\mathcal{N}\) be defined as
\begin{align*}
    E_r^\mathcal{N}(t_k, \vec{x}_k, \vec{u}^\mathcal{N}_k) \triangleq \sup_{t \in [t_k, t_{k+1})} \nrm{\Psi_r(\vec{x}(t), \vec{u}^\mathcal{N}_k) - \Psi_r(\vec{x}_k, \vec{u}^\mathcal{N}_k)}.
\end{align*}

\begin{assume}
    The functions \(\mathcal{L}_F \Psi_{r-1}\) and \(\mathcal{L}_G \Psi_{r-1}\) are Lipschitz continuous on \(D\) with respective constants \(\widehat{L}_{FH}, \widehat{L}_{GH}\).
\end{assume}

Similar to Lemmas \ref{lemma:E_bound} and \ref{lemma:adversarial_E}, it can be shown that
\begin{align*}
        E^\mathcal{N}(t_k, \vec{x}_k, \vec{u}_k) \leq \Delta_t (\widehat{L}_{FH} + \widehat{L}_{GH}b_u + c_{\alpha_r} L_H) (b_F + b_G b_u).
\end{align*}
We likewise define
\begin{align*}
    \delta_r(\Delta_t, c_{\alpha_r}) \triangleq \Delta_t (\widehat{L}_{FH} + \widehat{L}_{GH}b_u + c_{\alpha_r} L_H) (b_F + b_G b_u).
\end{align*}
The following subsets \(S_q \subset \R^n\) are defined in terms of the superlevel sets of each corresponding \(\Psi_q\) function:
\begin{subequations}
\label{eq:S_q_equations}
    \begin{align}
    S_q &\triangleq \{\vec{x} \in \R^n : \Psi_{q-1}(\vec{x}) \ggeq 0\},\ 1 \leq q \leq r, \\
    S_Q &\triangleq \bigcap_{q=1}^r S_q. \label{eq:S_Q}
    \end{align}
\end{subequations}
Similar to higher order methods for scalar-valued CBFs, these \(S_q\) sets form a hierarchy of sets involved with guaranteeing forward invariance of \(S_1 = \{x \in \R^n : H(x) \ggeq 0\}\). Intuitively, each set \(S_q\) corresponds to the states that ensure the previous set \(S_{q-1}\) remains forward invariant under \eqref{eq:multi_agent_dynamics}.

\begin{define}
    Let \(H : \R^n \to \mathbb{S}^p\) be a matrix valued function that is at least \(r\)-times continuously differentiable.
    Let \(S\) be defined as in \eqref{eq:safe_set}.
    Let \(D \subseteq \R^n\) be a domain such that \(S \subset D\).
    The function \(H\) is called a High-Order ARZE-MCBF for the dynamics \eqref{eq:multi_agent_dynamics} if there exists constants \(c_{\alpha_1}, \ldots, c_{\alpha_r} > 0\) such that for all \(\vec{x} \in D\) there exists a \(\vec{u}^\mathcal{N} \in \mathcal{U}^\mathcal{N}\) satisfying
    \begin{align}
        \Psi_r(\vec{x}, \vec{u}^\mathcal{N}) \ggeq \delta_r(\Delta_t, c_{\alpha_r}) I. \label{eq:disturbance_safety_cond}
    \end{align}
\end{define}
The set of feasible normal agent control inputs satisfying \eqref{eq:disturbance_safety_cond} is denoted
\begin{align}
    K_r^\mathcal{N} \triangleq \brc{\vec{u}^\mathcal{N} \in \mathcal{U}^\mathcal{N} : \Psi_r(\vec{x}, \vec{u}^\mathcal{N}) \ggeq \delta_r(\Delta_t, c_{\alpha_r}) I}. \label{eq:K_r_set}
\end{align}
Our final result presents conditions under which the sets \(S_q\) are rendered forward invariant under the dynamics \eqref{eq:multi_agent_dynamics}.
\begin{theorem}
    \label{thm:HOARZE-MCBF}
    Consider the dynamics \eqref{eq:multi_agent_dynamics} and the set \(S_Q\) in \eqref{eq:S_Q}. If \(H\) is a High-Order ARZE-MCBF, then any feedback control law for the normal agents \(\vec{\kappa}^\mathcal{N}(\cdot)\) satisfying \(\vec{\kappa}^\mathcal{N}(\vec{x}_k) \in K^\mathcal{N}_r(\vec{x}_k)\) for all \(t_k\), \(k \in \Nat\) renders the set \(S_Q\) forward invariant for the closed-loop system.
\end{theorem}

\begin{proof}
    Consider any interval \([t_k, t_{k+1})\) for \(k \in \Nat\).
    Suppose \(\vec{x}_k \in S_Q\). Choose any feedback control law satisfying \(\vec{\kappa}^\mathcal{N}(\vec{x}_k) \in K_r^\mathcal{N}(\vec{x}_k)\). It follows from \eqref{eq:K_r_set} that \(\Psi_r(\vec{x}_k, \vec{\kappa}(\vec{x}_k)) \ggeq \delta_r(\Delta_t, c_{\alpha_r}) I\). Using the definition of \(E^\mathcal{N}\) and similar arguments as in the proofs of Theorems \ref{thm:ZE-MCBF} and \ref{thm:ARZE-MCBF}, it can be shown that \(\Psi_r(\vec{x}(t), \vec{\kappa}(\vec{x}_k)) \ggeq 0\) for all \(t \in [t_k, t_{k+1})\).
    From \eqref{eq:Psi_high_degree}, this implies that \(\dot{\Psi}_{r-1}(\vec{x}(t)) + c_{\alpha_{r-1}} \Psi_{r-1}(\vec{x}(t)) \ggeq 0\) for all \(t \in [t_k, t_{k+1})\). By Proposition \ref{prop:mcbf} it follows that \(\Psi_{q-1}(\vec{x}(t)) \ggeq 0\) and \(\vec{x}(t) \in S_r\) on this time interval. This argument can be continued inductively for sets \(S_{r-2}, \ldots, S_1\). Observe that \(\vec{x}(t) \in S_1\) implies that \(\Psi_0(\vec{x}) = H(\vec{x}) \ggeq 0\) on the interval \([t_k, t_{k+1})\). By these arguments, we therefore have \(\vec{x}(t) \in S_Q\) for all \(t \in [t_k, t_{k+1})\). Since these arguments hold for arbitrary \(k \in \Nat\), the set \(S_Q\) is therefore forward invariant for all \(t \geq t_0\).
\end{proof}

Feasible control inputs within \(K^{\mathcal{N}}_r(\vec{x})\) can be computed using convex optimization methods similar to \eqref{eq:kappa_convex_program}.

\section{Conclusion}
\label{sec:conclusion}

This paper introduced methods to provide adversarial robustness guarantees to multi-agent Matrix Control Barrier Functions for sampled-data systems.
The paper derived novel theoretical results for set invariance under a zero-order control law and cooperative behavior among agents. We presented theoretical results for set invariance in the presence of adversarial inputs. Finally, the paper extended the results to systems having high relative degree.
Future work will explore integration of this theory in various multi-agent robotic settings to provide rigorous online safety guarantees in adversarial environments.

\section*{Acknowledgements}

Generative AI tools \cite{ChatGPT, Claude} were used to conduct mathematical brainstorming and identify relevant mathematical concepts. The paper itself was written solely by the author.

\bibliographystyle{IEEEtran}
\bibliography{bibliography}

\end{document}